\documentclass[10pt]{amsart}

\usepackage{graphicx} % Required for inserting images
\usepackage[utf8]{inputenc}
\usepackage[english]{babel}
\usepackage{braket}
\makeatletter
\def\l@subsection{\@tocline{2}{0pt}{2.5pc}{5pc}{}}
\makeatother
\usepackage{amssymb}
\usepackage{amsmath}
\usepackage{amsthm}
\usepackage{upgreek}
\usepackage{multicol}
\usepackage{enumerate}
\usepackage{float}
\usepackage{mathrsfs} %para letras cursivas
\usepackage[dvipsnames]{xcolor}
\usepackage{comment}
\usepackage{bbm}
\usepackage{amsthm} %para theoremstyle
\theoremstyle{definition}
\usepackage{amsmath,environ}
\NewEnviron{equsplit}
{
	\begin{equation*}\begin{split}
			\BODY
	\end{split}\end{equation*}
}
\newcommand{\R}{\mathbb{R}}

\newcommand{\C}{\mathbb{C}}
\newcommand{\T}{\mathbb{T}}

\newcommand{\G}{\mathcal{G}}

\usepackage{mathtools}
\renewenvironment{cases}{%
	\begin{dcases}%
	}{%
	\end{dcases}%\kern-\nulldelimiterspace%
}
\usepackage{aliascnt}
\usepackage[hidelinks]{hyperref}

\numberwithin{defi}{section}

\newcommand{\newautotheorem}[2]{%
	\newaliascnt{#1}{defi}%
	\newtheorem{#1}[#1]{#2}%
	\aliascntresetthe{#1}%
	\expandafter\def\csname #1autorefname\endcsname{#2}%
}

\newautotheorem{theorem}{Theorem}
\newautotheorem{proposition}{Proposition}
\newautotheorem{corollary}{Corollary}
\newautotheorem{definition}{Definition}
\newautotheorem{lemma}{Lemma}
\newautotheorem{example}{Example}
\newautotheorem{remark}{Remark}
\newautotheorem{nota}{Notation}

\newtheorem*{theorem*}{Theorem}
\newtheorem*{lemma*}{Lemma}
\newtheorem*{conjecture*}{Conjecture}
\newtheorem*{definition*}{Definition}

\newtheorem{mainthm}{Theorem}

\newcommand{\luno}{\ell^1(G\curvearrowright X)}

\renewcommand{\G}{\mathcal{G}}

\newcommand{\lone}[2]{\ell^1(#1\curvearrowright#2)}
\newcommand{\unit}[1]{\mathcal{U}(#1)}
\newcommand{\core}[1]{\text{core}(#1)}
\newcommand{\lpxg}[0]{\ell^p(X\times G)}
\newcommand{\Lp}[2]{F_\lambda^p(#1\curvearrowright#2)}
\renewcommand{\Re}[1]{\text{Re}(#1)}
\renewcommand{\Im}[1]{\text{Im}(#1)}
\newcommand{\action}[0]{G\curvearrowright X}
\newcommand{\actionof}[2]{#1\curvearrowright #2}
\newcommand{\Linfty}{L^\infty(X,\mu)}
\newcommand{\lunom}{\ell^1(G\curvearrowright (X,\mu))}

\usepackage[dvipsnames]{xcolor}
\usepackage[shortlabels]{enumitem}
\usepackage{tikz-cd}

\title{Isomorphisms of $\boldsymbol{\ell^1}$-crossed products of topological systems}
\date{\today}

\dedicatory{In memory of Fernando Abadie}

\author[Eusebio Gardella]{Eusebio Gardella}
\address{Eusebio Gardella,
	Department of Mathematical Sciences, Chalmers University of
	Technology and University of Gothenburg, Gothenburg SE-412 96, Sweden.}
\email{gardella@chalmers.se}
\urladdr{www.math.chalmers.se/~gardella}

\author[Tabaré Roland]{Tabaré Roland}
\address{Tabaré Roland, Instituto de Matemática y Estadística, Facultad de Ingeniería, Universidad de la República, Uruguay}
\email{tabarer@fing.edu.uy}

\begin{document}
	
	\begin{abstract}
		We give a complete description of the isometric isomorphisms between the
		$\ell^1$-crossed products of actions of discrete groups on
		nonempty compact Hausdorff spaces. In particular, two such
		crossed products are isometrically isomorphic as Banach
		algebras if and only if the actions are topologically isomorphic.
		It follows that $\lone{G}{X}$ remembers not just $X$ but even the acting group and the action, and is
		not determined by the transformation groupoid $G\ltimes X$ alone.
		The proof recovers $C(X)$ from the hermitian elements and uses
		the geometry of $L^p$-spaces to identify the invertible
		isometries. We also obtain analogous results for measurable
		actions.
	\end{abstract}

	\maketitle
	
%	\tableofcontents
	
\section{Introduction}

How much of a group action can be recovered from its crossed product?
The answer depends on the norm used to complete the algebraic crossed
product.  In this paper we consider an action $G\curvearrowright X$ of
a discrete group on a nonempty compact Hausdorff space and the Banach
algebra\footnote{There is also a canonical adjoint defined on $\lone{G}{X}$, which we however do not need for our analysis and therefore also do not introduce or use here.} $\lone{G}{X}$, obtained by completing the algebraic crossed
product in the norm
\[
\Big\|\sum_{s\in G} f_s\delta_s\Big\|_1
=\sum_{s\in G}\|f_s\|_\infty.
\]
Our main result gives a particularly clean answer, which we proceed to motivate.

There are reasons to expect the orbit structure of $G\curvearrowright X$ to be encoded in the crossed product, at least to some (weak) extent. Indeed, for minimal
$\mathbb Z$-actions on Cantor sets, Giordano, Putnam and Skau showed
that the associated $C^*$-crossed products are classified by strong
orbit equivalence; see
\cite{GPS}.  For topologically free actions, an isomorphism of reduced
$C^*$-crossed products that carries $C(X)$ onto $C(Y)$ corresponds,
under the usual hypotheses, to continuous orbit equivalence
\cite{Li}.  The analogous $L^p$-operator algebra crossed products,
for $p\neq 2$, detect their canonical copies of $C(X)$ intrinsically,
but their isometric isomorphism problem for topologically free actions
is again governed by continuous orbit equivalence; see
\cite{rigidityLp}.  These results make it natural to ask whether an
$\ell^1$-crossed product sees some form of
orbit equivalence, or a completely different set of data.

The group algebra case offers a clue, but not a complete answer.
When $X$ consists of one point, $\lone{G}{X}$ is $\ell^1(G)$, and
Wendel's theorem says that its isometric isomorphism class determines
$G$; see \cite{wendel}.  It is less clear what happens when the action has
nontrivial orbits: a crossed product combines the group with the
functions on $X$, and an isomorphism need not preserve the apparent
copies of either.  Earlier work on $\ell^1$-crossed products has
shown that their Banach algebra structure retains substantial
dynamical information.  For $\mathbb Z$-actions, de Jeu, Svensson
and Tomiyama proved, for example, that freeness can be characterized
by self-adjointness of closed ideals \cite{1}; see also
\cite{2,3,4,5} and the extension to more general groups in
\cite{preprint}.  

The answer is unexpectedly rigid: \emph{the isometric isomorphism
class of $\lone{G}{X}$ determines the entire action
$G\curvearrowright X$}. It remembers not only the orbits, but the
acting group and how that group acts on $X$, and it holds without any
freeness, minimality, amenability, or countability assumptions.

\begin{mainthm}[see \autoref{CorEquiv}]\label{mainThmA}
	Let $G\curvearrowright X$ and $H\curvearrowright Y$ be actions
	of discrete groups on nonempty compact Hausdorff spaces.
	Then $\lone{G}{X}$ and $\lone{H}{Y}$ are isometrically
	isomorphic as Banach algebras if and only if there are a group
	isomorphism $\psi\colon G\to H$ and a homeomorphism
	$\theta\colon X\to Y$ such that
	\[
	\theta(sx)=\psi(s)\theta(x)\]
for all $s\in G$ and $x\in X$.
\end{mainthm}

One remarkable consequence is that the isometric isomorphism class of
$\lone{G}{X}$ does not depend only on the transformation groupoid
$G\ltimes X$.  Indeed, let either $\mathbb Z/4\mathbb Z$ or
$(\mathbb Z/2\mathbb Z)^2$ act on itself by translation.  Both actions
are free and transitive, so both transformation groupoids are
isomorphic to the pair groupoid on four points. Their
$\ell^1$-crossed products, however, cannot be isometrically
isomorphic by \autoref{mainThmA}. Thus even the acting group,
which the pair groupoid has forgotten, survives in the norm of the
crossed product.

This also distinguishes the algebra studied here from the
$I$-norm completion of the compactly supported functions on
$G\ltimes X$ \cite{Renault80,austadOrtega}, despite the strong similarities between their constructions and the definitions of the norms. 
For the norm used here, one takes the supremum norm of each coefficient
$f_s$ and then sums over $s$. The groupoid $I$-norm instead sums along
each source or range fiber before taking the supremum over $x$.
This difference is what allows $\lone{G}{X}$ to retain information
that the transformation groupoid loses.

Our methods allow us to describe all isometric isomorphisms of $\ell^1(G\curvearrowright X)$:
in addition to an isomorphism $(\theta,\psi)$ of the actions, one may
multiply the image of each $\delta_s$ by a circle-valued continuous
function, and these functions satisfy a $1$-cocycle identity; see
\autoref{mainThm}. Although we do not formalize this, our arguments show that the isometric automorphism group 
of $\ell^1(G\curvearrowright X)$ is naturally isomorphic to the semidirect product of the 1-cocycle group $Z^1(G\curvearrowright X)$ and the automorphism group of $G\curvearrowright X$; in symbols 
\[\mathrm{Aut}(\ell^1(G\curvearrowright X))\cong Z^1(G\curvearrowright X)\rtimes \mathrm{Aut}(G\curvearrowright X).\] 
%The action of $(\theta,\psi)\in\operatorname{Aut}(\sigma)$ on $\omega\in Z^1(\sigma)$ is given by $((\theta,\psi)\cdot \omega)_s=\omega_{\psi^{-1}(s)}\circ\theta^{-1}$ for all $s\in G$. 

The proof uses an unexpected bridge to $L^p$-operator algebras.
The regular representation of $\lone{G}{X}$ acts on
$\ell^p(X\times G)$ for $p\in(1,\infty)\setminus\{2\}$, and
Lamperti's description of the surjective isometries of $\ell^p$
(see \cite{Lamperti58}) then gives a constraint on the invertible isometries of our
Banach algebra. This constraint implies that in the $L^p$ 
representation, an isometry can at
first sight appear to use different group elements on different pieces of
$X$. However, the $\ell^1$-norm rules out this possibility for an invertible
isometry of $\lone{G}{X}$: every such element has the form
$f\delta_s$, where $s\in G$ and $f\in C(X,\mathbb T)$.

The other ingredient is intrinsic to the Banach algebra.  Its
hermitian elements are exactly $C(X,\mathbb R)$, so its canonical
copy of $C(X)$ can be recovered without being specified in advance.
The quotient of the group of invertible isometries by
$C(X,\mathbb T)$ is then shown the be isomorphic to $G$, and conjugation on $C(X)$ recovers
the action up to topological isomorphism. The argument therefore combines the geometry of
$L^p$-spaces with a feature specific to the $\ell^1$-norm.
It also gives a proof of Wendel's theorem for discrete groups by
taking $X$ to be a point; see \autoref{rmkWendel}.

The topological result has a measurable counterpart.  For an action
of a countable discrete group on a standard measure space, we form
the corresponding crossed product with coefficients in
$L^\infty(X,\mu)$.  Passing to the Gelfand spectrum of
$L^\infty(X,\mu)$ allows us to apply \autoref{mainThmA}.  In the
probability-measure-preserving case, the canonical trace records the
measure as well as the action.

\begin{mainthm}[see \autoref{mainThmMeasurable} and
	\autoref{thmIsomMeasure}]\label{mainThmB}
	Let $G\curvearrowright (X,\mu)$ and
	$H\curvearrowright (Y,\nu)$ be measure-class-preserving actions
	of countable discrete groups on standard measure spaces.
	Then:
	\begin{enumerate}
		\item The actions are measurably isomorphic if and only if
		$\lone{G}{(X,\mu)}$ and $\lone{H}{(Y,\nu)}$ are isometrically
		isomorphic as Banach algebras.
		\item If the spaces are standard probability spaces and
		both actions preserve their measures, then the actions
		are measure-isomorphic if and only if their crossed
		products are isometrically isomorphic by an isomorphism
		intertwining the canonical traces $T_\mu$ and $T_\nu$.
	\end{enumerate}
\end{mainthm}

The paper is organized as follows.  In Section 2 we recall the
Banach algebra notions used to recover $C(X)$.  In Section 3 we
introduce the crossed products and the $L^p$ representation.
Section 4 proves the description of isometric isomorphisms and
\autoref{mainThmA}.  Section 5 establishes the measurable
versions in \autoref{mainThmB}.

	\subsection*{AI statement}
	All ideas and proofs in this work were obtained by the authors without the aid of AI/LLMs. A relatively final version of the manuscript was proofread by ChatGPT, who suggested a few minor clarifications to our arguments. AI was not used for any other purposes.

    \subsection*{Funding} The first named author was supported by the Swedish Research Council through the VR grant 2025-05763, ``Symmetries of simple $C^*$-algebras''.
    The second author was supported by a Comisión Académica de Posgrado scholarship. 
	
	\subsection*{Acknowledgments}
	This work was initiated while the second author was visiting the first one as part of his Masters studies in 2024. He would like to thank the Mathematics Department at the University of Gothenburg/Chalmers for the hospitality during his visit and the nice working environment.
    The second author thanks his supervisor Matilde Mart\'inez for the constant support during the development of this work. 

    \vspace{.2cm}
    
    Both authors are deeply grateful to Fernando Abadie, whose memory this paper is dedicated to. Fernando was the bachelor’s and master’s advisor of both authors, at different times, and the second author’s master’s project, from which this work grew, began under his supervision. His mathematical guidance, generosity, and care for his students have a lasting influence on both of us. \emph{Gracias por todo, Fer.}

	\section{Hermitian elements and invertible isometries}
	In this section, we introduce hermitian elements and invertible isometries of unital Banach algebras. 
	We also recall the definition of the \textit{core} of a unital Banach algebra from \cite{rigidityLp} and note that, under some assumptions on the algebra, it is a commutative $C^*$-algebra.

	\begin{definition}\label{defHerm}
		Let $A$ be a unital Banach algebra. 
		We say that an element $a\in A$ is \textit{hermitian} if $\lVert e^{iat}\rVert = 1$ for all $t\in \R$. 
		We denote the set of hermitian elements of $A$ as $A_\mathsf{h}$.
	\end{definition}

	The hermitian elements of a unital Banach algebra form a real Banach space and, if $A$ is a $C^*$-algebra, then $A_\mathsf{h}$ consists of the self-adjoint elements of $A$. 
We isolate the following particular case for later use.

    \begin{remark}
        If $X$ is a compact, Hausdorff space, then $C(X)_\mathsf{h} = C(X,\R).$
    \end{remark}
	
		The Vidav-Palmer theorem (see \cite[Theorem 6.9]{bonsallDuncan}) gives an abstract characterization of $C^*$-algebras using hermitian elements: a unital Banach algebra $A$ is a $C^*$-algebra if and only if $A=A_\mathsf{h} + iA_\mathsf{h}$. 
	This result suggests a way of finding a $C^*$-algebra inside a unital Banach algebra $A$. 
	Explicitly, given a unital Banach algebra $A$, we want $A_\mathsf{h}+i A_\mathsf{h}$ to be a $C^*$-subalgebra of $A$. 
	One can naturally define an involution on $A_\mathsf{h}+iA_\mathsf{h}$, but it is not always true that $A_\mathsf{h}+iA_\mathsf{h}$ is closed under multiplication.  
	This motivates the following definition.
	
	\begin{definition}\label{defCore}
		Let $A$ be a unital Banach algebra such that $A_\mathsf{h}$ is closed under multiplication. 
		We define the \textit{core of $A$} as $\core{A} = A_\mathsf{h} + i A_\mathsf{h}$.
	\end{definition}

	By \cite[Proposition 2.3]{rigidityLp}, $A_\mathsf{h}+iA_\mathsf{h}$ is an algebra whenever $A_\mathsf{h}$ is closed under multiplication, and in this case $\core{A}$ is a commutative, unital $C^*$-algebra.

	Let $A$ be a unital Banach algebra. 
    Denote the (Banach space) dual of $A$ by $A'$. 
	Following \cite[Definition 2.1 and Lemma 2.2]{bonsallDuncan}, we define the \textit{numerical range} $V(a)$ of $a\in A$ as 
	$$V(a) = \{f(a) \colon f\in A', f(1) = 1 = \lVert f\rVert\}\subseteq \C.$$ 
	By \cite[Page 15]{bonsallDuncan}, we have $V(z_1 a +z_2) = z_1 V(a) + z_2$ for all $a\in A$ and all $z_1,z_2\in \C$. 
	With the numerical range at our disposal, we can state another characterization of hermitian elements, from \cite{bonsallDuncan}, that will be of use later.
	\begin{lemma}[{\cite[Lemma 5.2]{bonsallDuncan}}]\label{charactHermitian}
		Let $A$ be unital Banach algebra, and let $a\in A$. 
		The following are equivalent:
		\begin{enumerate}[(1)]
			\item $a$ is an hermitian element.
			\item The numerical range $V(a)$ is contained in $\R$.
			\item The limit
			$$\lim_{t\to 0} \frac{1}{t} \left( \lVert 1+ita\rVert-1\right)$$
			exists and is equal to $0$, where the limit is taken as $t$ approaches $0$ through the real numbers.
		\end{enumerate}
	\end{lemma}

    We now introduce the notion of an invertible isometry of a unital normed algebra. 
    Although we work mostly with unital Banach algebras, we will later need this notion for a dense subalgebra of a Banach algebra rather than for the Banach algebra itself. 
    For this reason, we formulate the definition in the more general setting.

	\begin{definition}
		Let $A$ be a unital normed algebra. We define its group of \textit{invertible isometries} as 
		$$\mathcal{U}(A) = \{u\in A \colon u \text{ is invertible and } \lVert u \rVert,\lVert u^{-1} \rVert \leq 1\}.$$
		We refer to the elements of $\unit{A}$ as \textit{invertible isometries}.
	\end{definition}

    In the context of the definition above, it is easy to see that $\unit{A}$ is a group. 
	It is also straightforward to check that an invertible element $u\in A$ belongs to $\mathcal{U}(A)$ if and only if $\lVert u\rVert=1$ and $\lVert u^{-1} \rVert = 1$. 
	The following example will be needed later on.

    \begin{remark}
        For a compact, Hausdorff space $X$, one can readily verify that:
        $$\unit{C(X)} = C(X,\mathbb{T})= \{f\in C(X) \colon |f(x)| = 1 \text{ for all }x\in X\}.$$ 
    \end{remark}

	For a contractive unital homomorphism $\varphi\colon A\to B$ between unital normed algebras $A$ and $B$, it is clear that $\varphi(\unit{A})$ is contained in $\unit{B}$.

	\begin{lemma}\label{unitariesActByConj}
		Let $A$ be a unital Banach algebra. 
		If $u\in \unit{A}$ then $uA_\mathsf{h}u^{-1}$ is equal to $A_\mathsf{h}$. Consequently, $\unit{A}$ acts on $A_\mathsf{h}$ by conjugation.
	\end{lemma}
	\begin{proof}
        First we note that, for $u\in\unit{A}$, conjugation by $u$ is isometric: for all $a\in A$, we have that
        $$\lVert a\rVert = \lVert u^{-1}u a u^{-1}u\rVert \leq \lVert u^{-1}\rVert\lVert uau^{-1}\rVert \lVert u\rVert = \lVert uau^{-1}\rVert\leq \lVert u\rVert \lVert a \rVert \lVert u^{-1}\rVert = \lVert a \rVert,$$
        hence $\lVert uau^{-1}\rVert = \lVert a\rVert$. 
     	Let $a\in A_\mathsf{h}$ and let $u\in\unit{A}$. For all $t\in \R$, we have
		$$\frac{1}{t}\left(\lVert 1+it(uau^{-1}) \rVert -1\right) = \frac{1}{t} \left( \lVert u(1+ita)u^{-1} \rVert  -1\right) = \frac{1}{t} \left( \lVert 1+ita \rVert -1\right) \xrightarrow[t \to 0]{} 0.$$
		Hence, using the equivalence of (1) and (3) in \autoref{charactHermitian}, we obtain  $uau^{-1}\in A_\mathsf{h}$. 
        We deduce that $uA_\mathsf{h}u^{-1}\subseteq A_\mathsf{h}$, and by using $u^{-1}$ in place of $u$, we get equality. 
	\end{proof}

	To end this section, we take a detour in order to compute the value of the limit in (3) of \autoref{charactHermitian} when $A$ is a commutative $C^*$-algebra.
	In what follows, for a subset $E$ of $\C$, we denote its convex hull by $\mathrm{co}(E)$, that is, the smallest convex subset of $\C$ containing $E$.

	Let $A$ be a unital Banach algebra. 
	We say that $a\in A$ is \textit{normal} if there exist $h,k\in A_\mathrm{h}$ satisfying $hk=kh$ such that $a=h+ik$ (\cite[Definition 5.13]{bonsallDuncan}). 
	When $A$ is a $C^*$-algebra this is equivalent to $a^*a = aa^*$, so the terminology is not conflicting.

	\begin{lemma}\label{convexHull2}
		Let $X$ be a compact Hausdorff space and let $f\in C(X)$. Then $V(f)=\mathrm{co}(f(X))$.
	\end{lemma}
	\begin{proof}
		Since the spectrum of $f$ is $f(X)$ and $f$ is clearly normal, the result follows from \cite[Corollary 5.11]{bonsallDuncan}.
	\end{proof}

	The next result will be useful to characterize the hermitian elements of the $\ell^1$-crossed product (see \autoref{coreOfl1}). Given a complex number $z$, we denote by $\Re{z}$ its real part and by $\Im{z}$ its imaginary part.
	\begin{lemma}\label{limitLemma}
		Let $X$ be a compact Hausdorff space, and let $f\in C(X)$. Then
		\begin{equation*}
			\begin{split}
				&\lim_{t\to0^+} \frac{1}{t}\left( \lVert 1+itf\rVert - 1\right) = -\min \text{Im}(\mathrm{co}(f(X))),\\
				&\lim_{t\to0^-} \frac{1}{t}\left(\lVert 1+itf \rVert - 1\right) = -\max\text{Im}(\mathrm{co}(f(X))).
			\end{split}
		\end{equation*}
	\end{lemma}
	\begin{proof}
		Using \cite[Theorem 5.2]{bonsallDuncan} at the first step, we get 
		\begin{equation*}
			\begin{split}
				\lim_{t\to0^+}\frac{1}{t}\left( \lVert 1+itf\rVert-1\right) &= \max\{ \Re{z} \colon z \in V(if)\}\\
                &= \max\{ \Re{z} \colon z \in iV(f)\}\\
				&= \max\{-\Im{z} \colon z\in V(f)\}.
			\end{split}
		\end{equation*}
		By \autoref{convexHull2} we have $V(f) = \mathrm{co}(f(X))$, so we get
		\begin{equation*}
			\begin{split}
				\lim_{t\to0^+}\frac{1}{t}\left( \lVert 1+itf\rVert-1\right) &= \max\{-\Im{z} \colon z \in \mathrm{co}(f(X))\}\\
				&= -\min\{\Im{z} \colon z\in \mathrm{co}(f(X))\}\\
				&= -\min\Im{\mathrm{co}(f(X))}.
			\end{split}
		\end{equation*}
		Similarly, we get
		\begin{equation*}
			\begin{split}
				\lim_{t\to 0^-}\frac{1}{t}\left( \lVert 1+itf\rVert-1\right) &= -\lim_{t\to0^-} \frac{1}{-t}\left( \lVert 1+(-t)(-if)\rVert-1\right)\\
				&= -\lim_{t\to0^+}\frac{1}{t}\left( \lVert 1+t(-if)\rVert-1\right)\\
				&=-\max\{\Re{z} \colon z \in V(-if)\}\\
				&= -\max\{ \Re{z} \colon z \in -iV(f)\}\\
				&= -\max\{\Im{z} \colon z\in V(f)\}\\
				&= -\max\{\Im{z} \colon z \in \mathrm{co}(f(X))\}\\
				&= -\max \Im{\mathrm{co}(f(X))},
			\end{split}
		\end{equation*}
		which finishes the proof.
	\end{proof}

	\section{Crossed product Banach algebras} 
	We now introduce the two crossed product Banach algebras that will be of interest to us: the \textit{$\ell^1$-crossed product} and the \textit{reduced $L^p$-operator algebra crossed product}, both associated to the action of a discrete group on a compact Hausdorff space via homeomorphisms.

	Let $G\curvearrowright^\sigma X$ be an action of a discrete group $G$ on a compact, Hausdorff space $X$ by homeomorphisms. For $s\in G$, denote by $\sigma_s \colon X \to X$ the associated homeomorphism.  
	We denote by $G\curvearrowright^\alpha C(X)$ the associated action, where $s\in G$ acts via $\alpha_s(f) = f\circ \sigma_s^{-1}$ for all $f\in C(X)$. 
    When there is no risk of confusion we will drop the symbol $\sigma$ in the notation and simply write $\action$.

    \begin{definition}
        Let $\action$ be an action of a discrete group $G$ on a compact, Hausdorff space $X$. 
        Set
    	$$\lone{G}{X} = \left\{ f\colon G \to C(X) \text{ with } \sum_{s\in G} \lVert f(s)\rVert_\infty < \infty\right\},$$
    	where $\lVert\>\cdot\>\rVert_\infty$ denotes the supremum norm in $C(X)$.
        When endowed with the norm $\lVert f\rVert = \sum_{s\in G} \lVert f(s)\rVert_\infty$, for $f\in \luno$, it follows that $\lone{G}{X}$ is a Banach space. 
    	We define a product in $\lone{G}{X}$ the following way: for $f,g\in \lone{G}{X}$ and $s\in G$, we set
    	$$(fg)(s) = \sum_{t\in G} f(t)\alpha_t (g(t^{-1}s)).$$
    	With this structure, $\lone{G}{X}$ is a Banach algebra. 
        We refer to this algebra as the \textit{$\ell^1$-crossed product} associated to the action $\action$. 
    \end{definition}

	Given $s\in G$, we denote by $\delta_s$ the element of $\lone{G}{X}$ such that $\delta_s(t) = 0$ if $t\neq s$, and $\delta_s(s) = 1$ (the Dirac function at $s$). 
	An arbitrary element $f\in \lone{G}{X}$ can thus be written as 
	$$f= \sum_{s\in G} f(s) \delta_s = \sum_{s\in G}f_s\delta_s,$$
	where $f_s = f(s)$ for all $s\in G$. 
	We will usually use this description of elements of $\lone{G}{X}$ without explicit mention of it. With this notation, the product in $\luno$ is described by the rule
	$$(f_s\delta_s)(f_{t}\delta_{t}) = f_s \alpha_{s}(f_{t})\delta_{st}$$
	for all $s,t\in G$ and all $f_s,f_t \in C(X)$.

    \begin{remark}
        There is a canonical, isometric embedding $\iota\colon C(X)\to \luno$ given by $\iota(f) = f\delta_e$ for all $f\in C(X)$. For convenience, we usually omit $\iota$ and identify $f$ with $f\delta_e$.
    \end{remark}

    \begin{remark}\label{implAction}
        For $s\in G$, the element $\delta_s$ is an invertible isometry. 
        Moreover, if $f\in C(X)$, then $\delta_s f \delta_s^{-1} = \alpha_s(f).$
        In other words, $\delta_s$ implements $\alpha_s$ within $\luno$.
    \end{remark}

	For $p\in[1,\infty)$, we define the reduced $L^p$-operator algebra crossed product following \cite[Section 2]{twistedCrossProd}. (The original definition, due to Phillips in \cite{Phi_crossed}, is slightly different.)
	Recall that an \textit{$L^p$-operator algebra} is a Banach algebra $A$ such that there exist an $L^p$-space $E$ and an isometric homomorphism $A\to B(E)$.

    \begin{definition}\label{defLpCrossProd}
        For $p\in[1,\infty)$, let $\Lambda_p \colon \lone{G}{X} \to B(\ell^p(X\times G))$ be the map defined as follows: for $f = \sum_{s\in G}f_s\delta_s\in \lone{G}{X}$ we set
        \begin{equation*}
            \Lambda_p(f) \xi (x,t) = \sum_{s\in G} f_s(x)\xi(\sigma_s^{-1}(x),s^{-1}t)  
        \end{equation*}
        for all $\xi\in \lpxg$ and for all $(x,t)\in X\times G$. We define the \textit{reduced $L^p$-operator algebra crossed product} of $\action$, denoted by $\Lp{G}{X}$, as 
	$$\Lp{G}{X} = \overline{\Lambda_p(\lone{G}{X})}^{\, \|\cdot\|} \subseteq B(\lpxg).$$
    \end{definition}

    This is not the typical way to construct the reduced $L^p$-operator algebra crossed product. 
	Usually, the norm in $\Lp{G}{X}$ is defined by taking the supremum over the norm of the image of all regular representations of $\lone{G}{X}$; see \cite[Section 7]{modernLp} or \cite[Paragraph 2.17]{rigidityLp} for details. 
    Both constructions yield the same algebra by \cite[Proposition 3.1]{twistedCrossProd}.

    \begin{remark}
        Given $s\in G$, set $\lambda_s = \Lambda_p(\delta_s)\in B(\lpxg)$. It is routine to check that 
        $$\lambda_s(\xi)(x,t) = \xi(\sigma_s^{-1}(x),s^{-1}t)$$
        for all $\xi\in\lpxg$ and all $(x,t)\in X\times G$. With this notation, for $f=\sum_{s\in G}f_s \delta_s \in \luno$, we get
        $$\Lambda_p \left(\sum_{s\in G}f_s\delta_s \right) = \sum_{s\in G} f_s\lambda_s \in \Lp{G}{X}.$$
    \end{remark}

	We will need to describe the surjective linear isometries of $\lpxg$. 
    This, for $p\in(1,\infty)$ and $p\neq 2$, follows from the general Lamperti theorem. 
    We must introduce some notation first.

    \begin{nota}
        Given a (discrete) set $Y$ and a bijection $\varphi\colon Y \to Y$, we write $u_\varphi \in B(\ell^p(Y))$ for the invertible isometry given by $u_\varphi(\xi) = \xi \circ \varphi^{-1}$ for $\xi\in\ell^p(Y)$. Also, given $\eta\colon Y\to \T$, we write $m_\eta\in B(\ell^p(Y))$ for the invertible isometry given by $m_\eta(\xi) = \eta \xi$ for $\xi\in\ell^p(Y)$.
    \end{nota}

    Lamperti's theorem states that, for $p\in (1,\infty)\setminus\{2\}$, all surjective isometries are obtained from the isometries that we have just described. 
	We state the theorem in the generality needed for our context.
    
    \begin{theorem}[See, for example, {\cite[Theorem 3.7]{lamperti}}]\label{discreteLamperti}
		Let $p\in(1,\infty)$ with $p\neq2$ and let $Y$ be a set. 
        Let $T\colon \ell^p(Y)\to\ell^p(Y)$ be a surjective linear isometry. 
		Then there exist a bijection $\varphi\colon Y\to Y$ and a function $\eta\colon Y \to \T$ such that $T= m_\eta u_\varphi$.
	\end{theorem}

	\section{Rigidity of the $\ell^1$-crossed product}\label{proofSection}

    The purpose of this section is to prove the main result of this paper, \autoref{mainThm}, which asserts that two $\ell^1$-crossed products are isometrically isomorphic if and only if the actions are the topologically isomorphic; see \autoref{CorEquiv}. We define this notion first.

	\begin{definition}\label{isomTopSyst}
        We say that two actions $G\curvearrowright^\sigma X$ and $H\curvearrowright^\tau Y$ are \textit{topologically isomorphic} if there are a group isomorphism $\psi\colon G \to H$ and a homeomorphism $\theta\colon X\to Y$ such that $\theta\circ\sigma_s = \tau_{\psi(s)}\circ\theta$ for all $s\in G$. 
	\end{definition}

    The above notion is weaker than conjugacy since, even when $G=H$, we allow for a nontrivial automorphism of $G$.

	In what follows, for $(x,t)\in X\times G$, when writing $\mathbbm{1}_{(x,t)}$ we mean the characteristic function of $\{(x,t)\}$, which is clearly an element of $\lpxg$. 
	Similarly, if $Y$ is contained in $X\times G$, we denote the characteristic function of $Y$ by $\mathbbm{1}_{Y}$.
	\begin{lemma}\label{unitariesOnLp}
		Let $\action$ be an action of a discrete group $G$ on a compact, Hausdorff space $X$, and let $p\in(1,\infty)$ with $p\neq 2$. Let $\Lambda_p\colon \lone{G}{X} \to \Lp{G}{X}$ be the map from \autoref{defLpCrossProd}. 
		Then $\unit{\Lambda_p(\lone{G}{X})}$ is contained in the following set:
		$$\left\{ \sum_{s\in G} \eta\mathbbm{1}_{X_s} \lambda_s \in B(\lpxg)\colon \eta \text{ is a function } X\to \T \text{ and } X = \bigsqcup_{s\in G}X_s\right\}.$$
        We note that in the set above we do not require the function $\eta$ to be continuous in any way.
	\end{lemma}
	\begin{proof}
		For $s \in G$, let $\rho_s$ denote the operator in $B(\lpxg)$ given by
		$$\rho_s(\xi)(x,t) = \xi (x,ts)$$
        for all $\xi\in \lpxg$ and all $(x,t)\in X\times G.$
		
		\textbf{Claim:} \textit{For all $g\in G$ and all $a\in \Lp{G}{X}$ we have $\rho_g a = a \rho_g$.} 
        By density, it suffices to prove the claim for $a=\Lambda_p(f)$ with $f\in\luno$. 
		Let $f=\sum_{s\in G}f_s\delta_s\in\luno$ and fix $g\in G$.
        For $\xi\in\lpxg$ and $(x,t)\in X\times G$, we compute:
		\begin{equation*}
			\begin{split}
				(\rho_g\circ\Lambda_p(f))(\xi)(x,t) &= \rho_g \left( \sum_{s\in G}f_s\lambda_s \right)\xi (x,t) = \rho_g \left( \sum_{s\in G} f_s \lambda_s(\xi) \right)(x,t)\\
				&= \sum_{s\in G} f_s\lambda_s(\xi) (x,tg) = \sum_{s\in G}f_s(x) \lambda_s(\xi) (x,tg)\\
				&= \sum_{s\in G} f_s(x) \xi (\sigma_s^{-1}(x), s^{-1}tg).\\
			\end{split}
		\end{equation*}
        Similarly,
		\begin{equation*}
			\begin{split}
				(\Lambda_p(f)\circ\rho_g)(\xi)(x,t) &= \left( \sum_{s\in G}f_s\lambda_s \right) \rho_g(\xi)(x,t) = \sum_{s\in G}f_s \lambda_s (\rho_g(\xi)) (x,t)\\
				&= \sum_{s\in G}f_s(x) \lambda_s(\rho_g(\xi))(x,t)\\ &= \sum_{s\in G}f_s(x)\rho_g(\xi) (\sigma_s^{-1}(x), s^{-1}t)\\
				&= \sum_{s\in G}f_s(x)\xi(\sigma_s^{-1}(x),s^{-1}tg).
			\end{split}
		\end{equation*}
		This proves the claim.
        \vspace{0.2cm}

		Let $u\in \unit{\Lambda_p(\lone{G}{X})}$ be given. 
        Since $\Lambda_p$ is unital and contractive by \cite[Lemma 2.11]{twistedCrossProd}, it follows that $u$ is a surjective linear isometry of $\lpxg$. 
        Thus, by Theorem \ref{discreteLamperti} there exist a bijection $\varphi\colon X\times G\to X\times G$ and a function $\eta\colon X\times G\to \T$ such that $u = m_{\eta} u_\varphi.$ 
		We can write $\varphi=(\varphi_X,\varphi_G)$, where $\varphi_X$ and $\varphi_G$ are functions from $X\times G$ to $X$ and $G$, respectively. 
		Let $\xi\in \lpxg$ and $(x,t)\in X\times G$. Then
		\begin{equation*}
			\begin{split}
				(u\rho_s)(\xi)(x,t) &= u(\rho_s(\xi))(x,t)\\
                &= m_\eta u_\varphi (\rho_s (\xi))(x,t)\\
                &= \eta(x,t) \rho_s(\xi)(\varphi(x,t))\\
				&= \eta(x,t) \xi (\varphi_X(x,t),\varphi_G(x,t)s).\\
			\end{split}
		\end{equation*}
        Similarly,
        \begin{equation*}
			\begin{split}
				(\rho_s u(\xi))(x,t) &= \rho_s(u(\xi))(x,t)\\
                &= u(\xi)(x,ts)\\
                &= m_\eta u_\varphi (\xi)(x,ts)\\
                &= \eta(x,ts) \xi(\varphi(x,ts))\\
				&= \eta(x,ts)\xi(\varphi_X(x,ts), \varphi_G(x,ts)).
			\end{split}
		\end{equation*}
        By the claim above, we have that $u \rho_s = \rho_s u$ and thus
		\begin{equation}\tag{4.1}\label{eqCommutes}
			\eta(x,t) \xi (\varphi_X(x,t),\varphi_G(x,t)s) = \eta(x,ts) \xi(\varphi_X(x,ts),\varphi_G(x,ts)).
		\end{equation}
		Using that $\eta$ never vanishes, by taking $\xi = \mathbbm{1}_{\varphi(x,ts)}$ in  Equation (\refeq{eqCommutes}), we get that
		$$(\varphi_X (x,t),\varphi_G(x,t)s) = (\varphi_X(x,ts), \varphi_G(x,ts)),$$
		and it follows that $\eta(x,t)$ is equal to $\eta(x,ts)$. 
		Since $s\in G$ is arbitrary, we deduce that $\eta$ is actually only a function on $X$ (meaning that $\eta(x,t) = \eta(x,e)$ for all $t\in G$). 
		In the same way, $\varphi_X$ is actually only a function on $X$, and we also deduce that $\varphi_G(x,t) = \varphi_G(x,e)t$ for all $t\in G$. 
	 Fix $x\in X$ and let $s_x$ denote the inverse of $\varphi_G(x,e)$. Then for $t\in G$ we get $\varphi_G(x,t) = s_x^{-1}t$. 
		As the map $\{x\}\times G\to\{\varphi_X(x)\}\times G$ given by $(x,t)\mapsto (\varphi_X(x),s_x^{-1}t)$ is a bijection and coincides with the restriction of $\varphi$ to $\{x\}\times G$, it follows that $\varphi_X$ must be a bijection when regarding it as a function from $X$ to itself.

		As $u$ belongs to $\Lambda_p(\lone{G}{X})$, there are functions $f_s\in C(X)$, for $s\in G$, such that $u=\sum_{s\in G}f_s\lambda_s$.  
		Then, for all $\xi\in \lpxg$ and all $(x,t)\in X\times G$, we have 
		\begin{equation*}
			\eta(x) \xi (\varphi_X(x), s_x^{-1} t) = \sum_{s\in G}f_s \lambda_s(\xi)(x,t) = \sum_{s\in G} f_s(x) \xi (\sigma_s^{-1}(x), s^{-1}t).
		\end{equation*}
		Setting $\xi = \mathbbm{1}_{(\varphi_X(x),s_x^{-1}t)}$ in the above equation, we obtain  
		\begin{equation*}
			\eta(x) = \sum_{s\in G}f_s(x)\mathbbm{1}_{(\varphi_X(x),s_x^{-1}t)} (\sigma_s^{-1}(x),s^{-1}t) = f_{s_x}(x) \mathbbm{1}_{(\varphi_X(x), s_x^{-1}t)}(\sigma_{s_x}^{-1}(x), s_x^{-1}t).
		\end{equation*}
        Since $\eta(x)\neq 0$ it follows that $\varphi_X(x) = \sigma_{s_x}^{-1}(x)$ and $f_{s_x}(x) = \eta(x)$. 
        Setting $X_s = \{x\in X\colon s_x = s\}$, for $s\in G$, we obtain a partition $\{X_s\colon s\in G\}$ of $X$. 

        Given $x\in X_s$ we have $\varphi(x,t) = (\sigma_s^{-1}(x),s^{-1}t)$ for every $t\in G$.
        Using this, we can conclude that  
		\[u=\sum_{s\in G}f_s\lambda_s =\sum_{s\in G}\eta\mathbbm{1}_{X_s}\lambda_s.\qedhere\]
	\end{proof}

    \begin{lemma}\label{rmkLinearlyIndependent}
        Let $\action$ be an action of a discrete group $G$ on a compact, Hausdorff space $X$, and let $f_s\in\ell^{\infty}(X\times G)$, for $s\in G$, such that $\sum_{s\in G}\lVert f_s\rVert_\infty$ is finite. Then $\sum_{s\in G}m_{f_s}\lambda_s$ defines an element of $B(\lpxg)$. Furthermore, if $\sum_{s\in G}m_{f_s}\lambda_s = 0$, then $f_s=0$ for all $s\in G$.
    \end{lemma}
    \begin{proof}
        It is routine to check that $\sum_{s\in G}m_{f_s}\lambda_s$ is bounded in norm, hence it defines an element in $B(\lpxg)$. Assume that $\sum_{s\in G}m_{f_s}\lambda_s = 0$. Fix $r\in G$, and for $(x,t)\in X\times G$ take $\xi=\mathbbm{1}_{(\sigma_r^{-1}(x),r^{-1}t)}$. We have
        \begin{equation*}
            0 = \left(\sum_{s\in G}m_{f_s}\lambda_s \xi\right)(x,t) = \sum_{s\in G} f_s(x,t) \mathbbm{1}_{(\sigma_r^{-1}(x),r^{-1}t)} (\sigma_s^{-1}(x),s^{-1}t) = f_r(x,t).
		\end{equation*}
        This concludes the proof.
    \end{proof}

	The following proposition is a description of the invertible isometries of \mbox{$\lone{G}{X}$}.

	\begin{proposition}\label{unitariesOnl1}
		Let $\action$ be an action of a discrete group $G$ on a compact, Hausdorff space $X$. 
		Then
		$$\unit{\lone{G}{X}} = \{f\delta_s \colon  s\in G \text{ and } f\in C(X,\mathbb{T})\}.$$
	\end{proposition}
	\begin{proof}
		Let $u=\sum_{s\in G}f_s\delta_s$ be in $\unit{\lone{G}{X}}$. 
		Take $p\in (1,\infty)$, different from $2$, and consider the map $\Lambda_p \colon \lone{G}{X} \to \Lp{G}{X}$ as in \autoref{defLpCrossProd}, which is injective and contractive. 
        Then $\Lambda_p(\unit{\lone{G}{X}})$ is contained in \mbox{$\unit{\Lambda_p(\lone{G}{X})}$}, so it follows from \autoref{unitariesOnLp} that there are a function $f\colon X\to \T$ and a partition $\{X_s\colon s\in G \}$ of $X$ such that 
		$$ \sum_{s\in G} f \mathbbm{1}_{X_s} \lambda_s =  \Lambda_p \left(\sum_{s\in G}f_s\delta_s \right) = \sum_{s\in G} f_s \lambda_s.$$
        Using \autoref{rmkLinearlyIndependent} we get that $f_s$ equals $f \mathbbm{1}_{X_s}$ for all $s\in G$.
        Since $f_s$ is a continuous function and $|f_s|$ only takes the values $0$ and $1$, we deduce that $X_s$ must be a clopen subset of $X$, and it follows that $f$ must be continuous. 
		Thus $u = \sum_{s\in G}f\mathbbm{1}_{X_s} \delta_s$, and hence we obtain 
		$$\lVert u\rVert = \sum_{s\in G} \lVert f\mathbbm{1}_{X_s}\rVert_\infty = \# \{s\in G \colon X_s \neq \emptyset\},$$
        that is, the number of nonempty pieces $X_s$. 
		As $u$ belongs to $\unit{\lone{G}{X}}$, it must be $\lVert u\rVert=1$. 
		It follows that there is a unique $s\in G$ such that $X_s = X$, and $X_t$ is empty for all $t\in G\setminus\{s\}$. 
		In other words, we have $u = f\delta_s$, as desired.
	\end{proof}

    Next, we show how to recover the group $G$ from the Banach crossed product \mbox{$\luno$}.

    \begin{corollary}\label{recoverGrpFromAlg}
		Let $\action$ be an action of a discrete group $G$ on a compact, Hausdorff space $X$. The group $\unit{\lone{G}{X}}$ of invertible isometries of $\lone{G}{X}$ is (algebraically) isomorphic to the semidirect product $\unit{C(X)}\rtimes_\alpha G$. 
        Consequently, $G$ is isomorphic to the quotient $\unit{\lone{G}{X}}/\unit{C(X)}$.
	\end{corollary}
	\begin{proof}
        Using the description of $\unit{\lone{G}{X}}$ obtained in \autoref{unitariesOnl1}, it is directly verified that the map $\phi\colon \unit{\lone{G}{X}}\to\unit{C(X)}\rtimes_\alpha G$ given by $\phi(f\delta_s) = (f,s)$, for all $f\delta_s\in\unit{\lone{G}{X}}$, is an isomorphism of (discrete) groups.
	\end{proof}

    \begin{remark}\label{rmkWendel}
        \autoref{recoverGrpFromAlg} implies Wendel's theorem for discrete groups. 
        To see this, simply note that, for a discrete group $G$, we can consider the trivial action $G\curvearrowright\{\ast\}$ of $G$ on the one point space $\{*\}$, and that $\ell^1(G)$ is isomorphic to the crossed product $\lone{G}{\{\ast\}}$.
    \end{remark}
	
	Recall the notion of a hermitian element from \autoref{defHerm}, and the core of a unital Banach algebra from \autoref{defCore}.

	\begin{theorem}\label{coreOfl1}
		Let $\action$ be an action of a discrete group $G$ on a compact, Hausdorff space $X$. 
		Then the space of hermitian elements of the algebra $\lone{G}{X}$ is $C(X,\R)$. 
		Consequently, the core of $\lone{G}{X}$ is $C(X).$
	\end{theorem}
	\begin{proof}
		Let $f=\sum_{s\in G}f_s\delta_s\in \lone{G}{X}$ be a hermitian element of $\lone{G}{X}$. 
		We compute
		\begin{equation*}
			\begin{split}
				\lim_{t\to 0} \frac{1}{t}\left( \lVert 1+itf\rVert - 1\right) &= \lim_{t\to 0} \frac{1}{t}\left( \big\| (1+itf_e)\delta_e + \sum_{s\in G\setminus\{e\}}itf_s\delta_s\big\| -1\right)\\
				&=\lim_{t\to 0} \frac{1}{t}\left( \lVert 1+itf_e\rVert_\infty + \sum_{s\in G\setminus\{e\}} \lVert itf_s\rVert_\infty - 1\right)\\
				&= \lim_{t\to 0} \frac{1}{t}\left( \lVert 1+itf_e\rVert_\infty -1 \right) + \frac{|t|}{t}\left(\sum_{s\in G\setminus\{e\}}\lVert f_s\rVert_\infty \right)
			\end{split}
		\end{equation*}
		
		As $f_e$ belongs to $C(X)$, by \autoref{limitLemma} we get that
		\begin{equation}\label{eqLim1}
			\begin{split}
				\lim_{t\to 0^+} \frac{1}{t}\left( \lVert 1+itf\rVert - 1\right) = -\min\Im{co(f_e(X))} + \sum_{s\in G\setminus\{e\}} \lVert f_s\rVert_\infty.
			\end{split}
		\end{equation}
        Similarly,
		\begin{equation}\label{eqLim2}
			\begin{split}
				\lim_{t\to0^-} \frac{1}{t}\left( \lVert 1+itf\rVert - 1\right) = -\max\Im{co(f_e(X))} - \sum_{s\in G\setminus\{e\}}\lVert f_s\rVert_\infty.
			\end{split}
		\end{equation}
		Suppose that $\sum_{s\in G\setminus\{e\}}\lVert f_s\rVert_\infty > 0$.
        By \autoref{charactHermitian}, as $f$ is hermitian, we have that $\lim_{t\to 0} \frac{1}{t}(\lVert 1+itf\rVert-1)=0$. 
        Hence, combining the equalities in (2) and (3), we obtain that
        $$\max\Im{co(f_e(X))} - \min\Im{co(f_e(X))} = -2\left(\sum_{s\in G\setminus\{e\}} \lVert f_s\rVert_\infty\right).$$
        This is a contradiction, since the left-hand side of the equality is nonnegative whereas the right-hand side is strictly negative.
		Therefore, we conclude that \mbox{$\sum_{s\in G\setminus\{e\}}\lVert f_s\rVert_\infty=0$}. 
		Hence, $f = f_e\delta_e \in C(X)$, and as the embedding $C(X)\to\luno$ is isometric we have that $f_e$ is an hermitian element in $C(X)$. 
		Using that the hermitian elements of $C(X)$ are the continuous functions from $X$ to $\R$, we deduce that $f\in C(X,\mathbb{R})$.
		
		As $C(X,\R)$ is clearly closed under multiplication, it follows from \cite[Proposition 2.3]{rigidityLp} that $\core{\lone{G}{X}}$ is well defined, and is equal to $C(X,\R) + iC(X,\R)$, that is, $C(X)$.
	\end{proof}

	By \autoref{unitariesActByConj} we have that the invertible isometries of a unital Banach algebra act on its core by conjugation. 
	At this point, we know both the invertible isometries of $\lone{G}{X}$ and the core of $\lone{G}{X}$. 
	It is clear that $\unit{C(X)}$ acts trivially on $\core{\lone{G}{X}}$, as the core is equal to $C(X)$, which is commutative, and the action is by conjugation. 
	By Gelfand duality, this means that we actually have an action of the quotient group $\unit{\lone{G}{X}}/\unit{C(X)}$ on $X$.

    \begin{proposition}\label{recoverActionFromAlg}
		Let $\action$ be an action of a discrete group $G$ on a compact, Hausdorff space $X$. 
		Then the action of the quotient group $\unit{\lone{G}{X}}/\unit{C(X)}$ on $X$ and the action of $G$ on $X$ are topologically isomorphic.
	\end{proposition}
	\begin{proof}
		For all $s\in G$, we have that $\delta_s$ belongs to $\unit{\lone{G}{X}}$, and we denote its image in the quotient group $\unit{\lone{G}{X}}/\unit{C(X)}$ by $\overline{\delta}_s$. 
        It follows from \autoref{recoverGrpFromAlg} that every element of $\unit{\lone{G}{X}}/\unit{C(X)}$ is represented by a unique element of the form $\overline{\delta}_s$, and that the map $\overline{\delta}$ from $G$ to $\unit{\lone{G}{X}}/\unit{C(X)}$ given by $s\mapsto\overline{\delta}_s$, for all $s\in G$, is an isomorphism. 
		Moreover, for all $f\in C(X)$, using \autoref{implAction} at the last step, we have that 
        $$\overline{\delta}_s\cdot f = \delta_s f \delta_s^{-1} = \alpha_s(f).$$
        We conclude that the actions are isomorphic, as desired.
	\end{proof}

    We are now ready prove the main result of this paper, in which we describe all isometric isomorphisms between $\ell^1$-crossed products using 1-cocycles and topological isomorphisms. 

    \begin{theorem}\label{mainThm}
		Let $G\curvearrowright X$ and $H\curvearrowright Y$ be actions of discrete groups $G$ and $H$ on compact Hausdorff spaces $X$ and $Y$, and let $G\curvearrowright^\alpha C(X)$ and \mbox{$H\curvearrowright^\beta C(Y)$} be the respective induced actions. 
		Let \mbox{$\Phi \colon \lone{G}{X}\to\lone{H}{Y}$} be a linear map.
        Then $\Phi$ is an isometric isomorphism if and only if
        there are 
        \begin{enumerate}
            \item a topological isomorphism $(\theta,\psi)$ between $G\curvearrowright X$ and $H\curvearrowright Y$, as in \autoref{isomTopSyst}, and
            \item a map $\omega\colon G\to C(Y,\T)$ satisfying
            \begin{enumerate}
                \item $\omega_e = 1$ and
                \item $\omega_{st} = \omega_s \beta_{\psi(s)} (\omega_t)$ for all $s,t\in G$,
            \end{enumerate}\end{enumerate}
        such that for all $f\in C(X)$ and $s\in G$ we have
           \[\tag{4.2}\label{eq:Phi}\Phi(f\delta_s) = (f\circ\theta^{-1})\omega_s\delta_{\psi(s)}.\]
    \end{theorem}
    \begin{proof}
        Assume that $\Phi$ is an isometric isomorphism. 
        Then $\Phi$ restricts to an isomorphism between $\core{\lone{G}{X}}$ and $\core{\lone{H}{Y}}$ which, by \autoref{coreOfl1}, yields an isometric isomorphism between $C(X)$ and $C(Y)$. 
        By Gelfand duality, there is a homeomorphism $\theta\colon X\to Y$ such that $\Phi(f) = f\circ\theta^{-1}$ for all $f\in C(X)$.
        Since $\Phi$ is isometric, it induces a group isomorphism between $\unit{\lone{G}{X}}$ and $\unit{\lone{H}{Y}}$ which by \autoref{recoverGrpFromAlg} induces an isomorphism $\psi\colon G\to H$ and a map $\omega\colon G\to C(Y,\T)$ such that 
        $$\Phi(f\delta_s) = (f\circ\theta^{-1})\omega_s\delta_{\psi(s)}$$ 
        for all $f\in C(X)$ and $s\in G$, showing (4.1)
        It is clear that $\omega_e = 1$. Using
        that $\Phi$ is a homomorphism, for $s,t\in G$ we have 
        $$\omega_{st}\delta_{\psi(st)} = \Phi(\delta_{st}) = \Phi(\delta_s)\Phi(\delta_t) = \omega_s\delta_{\psi(s)}\omega_{t}\delta_{\psi(t)} = \omega_s \beta_{\psi(s)}(\omega_t) \delta_{\psi(st)},$$
        hence $\omega_{st} = \omega_s \beta_{\psi(s)}(\omega_t)$.
        This proves that (2) holds.
        To verify (1), let $f\in C(X)$ and $s\in G$. Then
        \begin{equation*}
            \begin{split}
                \alpha_s(f)\circ\theta^{-1} &= \Phi(\alpha_s(f)) = \Phi(\delta_s f \delta_s^{-1}) =\Phi(\delta_s)\Phi(f)\Phi(\delta_s^{-1})\\
                &= (\omega_s\delta_{\psi(s)})(f\circ\theta^{-1})(\omega_{s^{-1}}\delta_{\psi(s)^{-1}})\\
                &=(\omega_s\beta_{\psi(s)}(f\circ\theta^{-1})\delta_{\psi(s)})(\omega_{s^{-1}}\delta_{\psi(s)^{-1}})\\
                &=\omega_s\beta_{\psi(s)}(f\circ\theta^{-1})\beta_{\psi(s)}(\omega_{s^{-1}}) \delta_e\\
                &=\beta_{\psi(s)}(f\circ\theta^{-1}) \omega_s\beta_{\psi(s)}(\omega_{s^{-1}})\\
                &=\beta_{\psi(s)}(f\circ\theta^{-1})\omega_e = \beta_{\psi(s)}(f\circ\theta^{-1}).
            \end{split}
        \end{equation*}
       This shows that $(\theta,\psi)$ is a topological isomorphism between $G\curvearrowright X$ and $H\curvearrowright Y$, as desired. 

        For the converse, assume that $(\theta,\psi)$ and $\omega$ as in the statement are given, and that $\Phi$ satisfies (4.1). 
        Then $\Phi$ is clearly isometric and bijective, and its inverse is obtained from $(\theta^{-1},\psi^{-1})$ and the conjugate of $\omega$. Let us verify that $\Phi$ is a homomorphism. Let $f,g\in C(X)$ and $s,t\in G$. We have
        $$\Phi((f\delta_s)(g\delta_t)) = \Phi(f\alpha_s(g)\delta_{st}) = (f\circ\theta^{-1}) (\alpha_s(g) \circ \theta^{-1}) \omega_{st}\delta_{\psi(st)}.$$
        On the other hand,
        \begin{equation*}
            \begin{split}
                \Phi(f\delta_s)\Phi(g\delta_t) &= ((f\circ\theta^{-1})\omega_s\delta_{\psi(s)})((g\circ\theta^{-1})\omega_t\delta_{\psi(t)})\\
                &= (f\circ\theta^{-1})\omega_s (\beta_{\psi(s)}((g\circ\theta^{-1}) \omega_t))\delta_{\psi(s)\psi(t)}\\
                &=(f\circ\theta^{-1})(\beta_{\psi(s)}(g\circ\theta^{-1})) \omega_s\beta_{\psi(s)}(\omega_t)\delta_{\psi(st)}.
            \end{split}
        \end{equation*}
        Using (1), we have $\beta_{\psi(s)}(g\circ\theta^{-1}) = \alpha_s(g)\circ\theta^{-1}$, and from (2) we obtain $\omega_s\beta_{\psi(s)}(\omega_t) = \omega_{st}$. Thus
        $$\Phi(f\delta_s)\Phi(g\delta_t) = (f\circ\theta^{-1})(\alpha_s(g)\circ\theta^{-1})\omega_{st}\delta_{\psi(st)}.$$
        Hence $\Phi(f\delta_s)\Phi(g\delta_t) = \Phi((f\delta_s)(g\delta_t))$. This finishes the proof.
    \end{proof}

We isolate the main consequence of \autoref{mainThm}.

    \begin{corollary}\label{CorEquiv}
        Let $G\curvearrowright X$ and $H\curvearrowright Y$ be actions of discrete groups $G$ and $H$ on compact Hausdorff spaces $X$ and $Y$, respectively. Then the following are equivalent:
        \begin{enumerate}\item[(i)] $\ell^1(G\curvearrowright X)$ is isometrically isomorphic to $\ell^1(H\curvearrowright Y)$;
        	\item[(ii)] $G\curvearrowright X$ is topologically isomorphic to $H\curvearrowright Y$.\end{enumerate}  
    \end{corollary}

    \section{Measurable and measure-preserving actions}

    In this section, we introduce the $\ell^1$-crossed product $\lone{G}{(X,\mu)}$ associated to a measure-class-preserving action $\actionof{G}{(X,\mu)}$. We assume throughout this section that the groups acting are countable. 
    We prove rigidity results for actions on standard spaces (see \autoref{mainThmMeasurable}) and for probability-measure-preserving actions on standard probability spaces (see \autoref{thmIsomMeasure}).
    
    We recall that, for measure spaces $(X,\mu)$ and $(Y,\nu)$, a measurable map $S\colon X \to Y$ is \textit{measure-class-preserving} if
    $$\mu(S^{-1}(E)) = 0 \text{ if and only if } \nu(E) = 0,$$
    for every measurable subset $E\subseteq Y$.  A measure-class-preserving map is said to be \textit{invertible} if it is invertible as a map $X\to Y$ and its inverse is also measurable. 
    For an invertible measure-class-preserving map, its inverse is automatically measure-class-preserving.

    Let $(X,\mu)$ be a measure space and let $G\curvearrowright(X,\mu)$ be a measure-class-preserving action. 
    For $s\in G$, let $\sigma_s\colon X\to X$ be the associated map. 
    Then $G\curvearrowright (X,\mu)$ induces an action on the von Neumann algebra $\Linfty$ using duality.

    \begin{definition}
        Let $\actionof{G}{(X,\mu)}$ be a measure-class-preserving action of a countable, discrete group $G$ on a measure space $(X,\mu)$. 
        Set
        $$\lunom = \left\{ f\colon G \to \Linfty \text{ with } \sum_{s\in G} \lVert f(s)\rVert_\infty < \infty \right\}.$$
        When endowed with the norm $\lVert f \rVert = \sum_{s\in G}\lVert f(s)\rVert_\infty$, for $f\in \lunom$, it follows that $\lunom$ is a Banach space. We define a product in $\lunom$ the following way: for $f,g\in\lunom$ and $s\in G$, we set 
        $$(fg)(s) = \sum_{t\in G} f(t)\alpha_t(g(t^{-1}s)).$$
        With this structure $\lunom$ is a Banach algebra. 
    \end{definition}

    \begin{definition}\label{isomMeasurableSyst}
        We say that two actions $G\curvearrowright^\sigma(X,\mu)$ and $H\curvearrowright^\tau(Y,\nu)$ are \textit{measurably isomorphic} if there are
        \begin{enumerate}[(1)]
            \item a $\mu$-null, $G$-invariant subset $M\subseteq X$,
            \item a $\nu$-null, $H$-invariant subset $N\subseteq Y$,
            \item an invertible measure-class-preserving map $\theta\colon X\setminus M\to Y\setminus N$, and
            \item a group isomorphism $\psi\colon G\to H$
        \end{enumerate}
        such that $\theta(\sigma_s(x)) = \tau_{\psi(s)}(\theta (x))$ for all $x\in X\setminus M$ and all $s\in G$.
    \end{definition}

    In what follows, by a \textit{measurable isomorphism} between measurable spaces $X$ and $Y$ we mean a measurable map $f\colon X\to Y$ that is bijective and its inverse is also measurable. 
    Two measure spaces $(X,\mu)$ and $(Y,\nu)$ are \textit{measurably isomorphic} if there is a measurable isomorphism $f\colon X\to Y$ between them. 
    
    We recall that a \textit{Polish space} is a separable completely metrizable topological space. 
    \begin{definition}\label{stdMeasure}
        A \textit{standard measure space} is a measure space $(X,\mu)$ such that $\mu$ is a $\sigma$-finite measure and there is a measurable subset $E\subseteq X$, with $\mu(X\setminus E) = 0$, which is measurably isomorphic to a measurable subset of a Polish space (endowed with the relative measurable structure).
    \end{definition}

    We note that if $G\curvearrowright (X,\mu)$ is an action and $N\subseteq X$ is a $\mu$-null subset, then the $G$-invariant subset $ G\cdot N =\bigcup_ {s\in G}\sigma_s(N)$ is again $\mu$-null, because the action is measure-class-preserving and $G$ is countable.

    Not surprisingly, isomorphism for measurable systems can be characterized in terms of the algebras of essentially bounded functions. 
    We include the details for the convenience of the reader, since they show why one needs to discard sets of measure zero. 
    \begin{lemma}\label{measurableIsomOnAlg}
        Let $G\curvearrowright^\sigma (X,\mu)$ and $H\curvearrowright^\tau (Y,\nu)$ be measure-class-preserving actions of countable, discrete groups $G$ and $H$ on standard measure spaces $(X,\mu)$ and $(Y,\nu)$, respectively. 
        Let \mbox{$G\curvearrowright^\alpha \Linfty$} and $H\curvearrowright^\beta L^\infty(Y,\nu)$ be the induced actions. 
        Then $G\curvearrowright^\sigma (X,\mu)$ and $H\curvearrowright^\tau(Y,\nu)$ are measurably isomorphic if and only if there are a group isomorphism $\psi\colon G\to H$ and an isomorphism $\varphi\colon\Linfty\to L^\infty(Y,\nu)$ such that $\varphi\circ\alpha_s = \beta_{\psi(s)}\circ\varphi$ for all $s\in G$.
    \end{lemma}
    \begin{proof}
        If $G\curvearrowright^\sigma (X,\mu)$ and $H\curvearrowright^\tau (Y,\nu)$ are measurably isomorphic, there are a group isomorphism $\psi\colon G \to H$, subsets $M\subseteq X$, $N\subseteq Y$ and a measure-class-preserving map $\theta\colon X\setminus M\to Y\setminus N$ as in \autoref{isomMeasurableSyst}. 
        For $f\in\Linfty$, we define $\varphi(f)(y) = (f \circ \theta^{-1})(y)$ for all $y\not\in N$. 
        As $\nu(N) = 0$ and $\theta$ is measure-class-preserving, this defines a map $\varphi \colon \Linfty\to L^\infty(Y,\nu)$, which is an isomorphism given that $\mu(M) = 0$. 
        It is routine to check that $\varphi\circ\alpha_s = \beta_{\psi(s)}\circ\varphi$ for all $s\in G$.

        For the converse, use \cite[Corollary I.38]{CrossedProductsDPWilliams} to find a $\mu$-null set $M_0\subseteq X$, a $\nu$-null set $N_0\subseteq Y$, and a measurable isomorphism $\theta_0\colon X\setminus M_0\to Y\setminus N_0$ such that $\varphi(f) = f\circ \theta_0^{-1}$ for all $f\in \Linfty$. 
        For $s\in G$, let $E_s$ be the set of points in $X$ where the equality $\theta_0 \circ \sigma_s = \tau_{\psi(s)}\circ\theta_0$ fails. 
        Since $\varphi \circ \alpha_s = \beta_{\psi(s)}\circ\varphi$, each $E_s$ is a $\mu$-null set and, as $G$ is countable, the union of the $E_s$ is again $\mu$-null. 
        Let $M$ be the smallest $G$-invariant, $\mu$-null set containing $E_s$ and $M_0$. 
        Finally, let $N$ be the $H$-invariant, $\nu$-null subset of $Y$ such that $\theta_0$ restricts to a bijection $\theta\colon X\setminus M \to Y\setminus N$. 
        Then $\theta$ is an invertible measure-class-preserving map satisfying $\theta\circ \sigma_s = \tau_{\psi(s)}\circ \theta$ for all $s\in G$.
    \end{proof}

    We will need the following lemma both in the measure-class-preserving and in the probability-measure-preserving settings.

    \begin{lemma}\label{lemaMeasureIsom}
        Let $G\curvearrowright (X,\mu)$ and $H\curvearrowright (Y,\nu)$ be measure-class-preserving actions of countable, discrete groups $G$ and $H$ on standard measure spaces $(X,\mu)$ and $(Y,\nu)$, respectively, and let \mbox{$G\curvearrowright^\alpha \Linfty$} and $H\curvearrowright^\beta L^\infty(Y,\nu)$ be the induced actions. 
        Suppose that $\Phi\colon\lone{G}{(X,\mu)}\to\lone{H}{(Y,\nu)}$ is an isometric isomorphism. 
        Then $\varphi = \Phi|_{\Linfty}$ is an isomorphism between $\Linfty$ and $L^\infty(Y,\nu)$, and there exists a group isomorphism $\psi\colon G\to H$ such that $\varphi \circ \alpha_s = \beta_{\psi(s)}\circ\varphi$ for all $s\in G$.
    \end{lemma}
    \begin{proof}
        Let $\widehat{X}$ and $\widehat{Y}$ be the Gelfand spectra of $\Linfty$ and $L^\infty(Y,\nu)$, respectively,
        and denote by $\G_X\colon \Linfty\to C(\widehat{X})$ and $\G_Y\colon L^\infty(Y,\nu)\to C(\widehat{Y})$ the respective Gelfand transforms. 
        Let $G\curvearrowright^{\widehat\alpha}C(\widehat{X})$ and $H\curvearrowright^{\widehat\beta}C(\widehat{Y})$ denote the induced actions. 
        
        Let $\Phi_X\colon \lone{G}{(X,\mu)}\to\lone{G}{\widehat{X}}$ be the map given by $\Phi_X(f\delta_s) = \G_X(f)\delta_s$ for all $f\in\Linfty$ and $s\in G$.
        Similarly, let $\Phi_Y\colon \lone{H}{(Y,\nu)}\to\lone{H}{\widehat{Y}}$ be the map given by $\Phi_Y(g\delta_t) = \G_Y(g)\delta_t$ for all $g\in L^\infty(Y,\nu)$ and $t\in H$. 
        It is clear that both $\Phi_X$ and $\Phi_Y$ are isometric isomorphisms, hence the map $\widehat{\Phi}\colon\lone{G}{\widehat{X}}\to\lone{H}{\widehat{Y}}$ defined as $\widehat{\Phi} = \Phi_Y\circ\Phi\circ\Phi_X^{-1}$ is an isometric isomorphism. 
        By \autoref{mainThm}, the restriction $\widehat\varphi = \Phi|_{C(\widehat{X})}\colon C(\widehat{X})\to C(\widehat{Y})$ is an isomorphism, and there exists a group isomorphism $\psi\colon G\to H$ such that $\widehat\varphi\circ\widehat\alpha_s = \widehat\beta_{\psi(s)}\circ\widehat\varphi$ for all $s\in G$.
        Consequently, $\varphi=\G_Y^{-1}\circ\widehat\varphi \circ \G_X$, which is the restriction of $\Phi$ to $\Linfty$, is an isomorphism between $\Linfty$ and $L^\infty(Y,\nu)$.
        It is directly verified that $\varphi\circ\alpha_s = \beta_{\psi(s)}\circ\varphi$ for all $s\in G$.  
    \end{proof}

    We obtain an analogous result to \autoref{CorEquiv} in the setting of measure-class-preserving actions.

    \begin{theorem}\label{mainThmMeasurable}
        Let $\actionof{G}{(X,\mu)}$ and $\actionof{H}{(Y,\nu)}$ be measure-class-preserving actions of discrete, countable groups $G$ and $H$ on standard measure spaces $(X,\mu)$ and $(Y,\nu)$, respectively. 
        Then \mbox{$\lone{G}{(X,\mu)}$} and $\lone{H}{(Y,\nu)}$ are isometrically isomorphic if and only if $\actionof{G}{(X,\mu)}$ and $\actionof{H}{(Y,\nu)}$ are measurably isomorphic.
    \end{theorem}
    \begin{proof}
        The ``only if'' implication follows from \autoref{lemaMeasureIsom} and \autoref{measurableIsomOnAlg}.

        Conversely, suppose that $\actionof{G}{(X,\mu)}$ and $\actionof{H}{(Y,\nu)}$ are isomorphic measurable systems.
        By \autoref{measurableIsomOnAlg}, there are a group isomorphism $\psi\colon G\to H$ and an isomorphism $\varphi\colon\Linfty\to L^\infty(Y,\nu)$ such that $\varphi\circ\alpha_s = \beta_{\psi(s)}\circ\varphi$ for all $s\in G$.
        Let $\Phi\colon \lone{G}{(X,\mu)}\to\lone{H}{(Y,\nu)}$ be the map given by $\Phi(f\delta_s) = \varphi(f)\delta_{\psi(s)}$, for all $f\in\Linfty$ and $s\in G$. It is routine to check that $\Phi$ is an isometric isomorphism between the algebras.     
    \end{proof}

We need one more notion of equivalence for dynamical systems, this time taking measures into account.
     
    \begin{definition}\label{isomMeasureSyst}
        We say that two measure-preserving actions $G\curvearrowright^\sigma(X,\mu)$ and $H\curvearrowright^\tau(Y,\nu)$ are \textit{measure-isomorphic} if they are measurably isomorphic, and the map $\theta$ in \autoref{isomMeasurableSyst} can be chosen to be measure-preserving. 
    \end{definition}

    A \textit{standard probability space} is a probability space $(X,\mu)$ that is measurable isomorphic to a subset of a Polish space. 
    Every standard probability space is a standard measure space as in \autoref{stdMeasure}.
    If $(X,\mu)$ is a standard probability space, then $\mu$ induces a unital, contractive linear functional $I_\mu$ on $\Linfty$ given by
    $I_\mu(f)= \int_X f \ d\mu$ for all $f\in\Linfty$.
    An action $G\curvearrowright^\sigma(X,\mu)$ is said to be a \textit{probability-measure-preserving action} if, for all measurable subsets $E\subseteq X$, we have $\mu (E) = \mu(\sigma_s(E))$ for all $s\in G$. In this case, $I_\mu$ is $G$-invariant.

    \begin{lemma}\label{measureIsomOnAlg}
        Let $G\curvearrowright^\sigma (X,\mu)$ and $H\curvearrowright^\tau (Y,\nu)$ be probability-measure-preserving actions of countable, discrete groups $G$ and $H$ on standard probability spaces $(X,\mu)$ and $(Y,\nu)$, respectively, and let $G\curvearrowright^\alpha \Linfty$ and $H\curvearrowright^\beta L^\infty(Y,\nu)$ be the induced actions. 
        Then $G\curvearrowright^\sigma (X,\mu)$ and $H\curvearrowright^\tau(Y,\nu)$ are measure-isomorphic if and only if there are a group isomorphism $\psi\colon G\to H$ and an isomorphism $\varphi\colon\Linfty\to L^\infty(Y,\nu)$ such that $I_\nu\circ\varphi=I_\mu$ and $\varphi\circ\alpha_s = \beta_{\psi(s)}\circ\varphi$ for all $s\in G$.
    \end{lemma}
    \begin{proof}
        Assume that $G\curvearrowright^\sigma (X,\mu)$ and $H\curvearrowright^\tau(Y,\nu)$ are measure-isomorphic. Choose a group isomorphism $\psi\colon G\to H$, subsets $M\subseteq X$ and $N\subseteq Y$, and an isomorphism $\theta\colon X\setminus M\to Y\setminus N$ as in \autoref{isomMeasureSyst}. 
        For all $y\not\in N$ we define $\varphi(f)(y) = (f\circ \theta^{-1})(y)$. 
        As $\nu(N) = 0$, this defines a map $\varphi\colon\Linfty\to L^\infty(Y,\nu)$, which is an isomorphism since $\mu(M) = 0$. 
        It is routine to check that $\varphi\circ\alpha_s = \beta_{\psi(s)}\circ\varphi$ for all $s\in G$. 
        For any measurable subset $E\subseteq X\setminus M$ we have $\nu(\theta(E)) = \mu(E)$, hence $I_\nu(\varphi(\mathbbm{1}_E)) = I_\mu(\mathbbm{1}_E)$. 
        As the linear span of the characteristic functions is dense, we conclude that $I_\nu\circ\varphi = I_\mu$. 

        For the converse, by \cite[Corollary I.38]{CrossedProductsDPWilliams} there are a $\mu$-null set $M_0\subseteq X$, a $\nu$-null set $N_0\subseteq Y$ and a measurable isomorphism $\theta_0\colon X\setminus M_0\to Y\setminus N_0$ such that $\varphi(f) = f\circ\theta_0^{-1}$. For $s\in G$, let $E_s$ be the set of points in $X$ where the equality $\theta_0\circ \sigma_s = \tau_{\psi(s)}\circ \theta_0$ fails, which is a $\mu$-null set since $\varphi\circ\alpha_s = \beta_{\psi(s)}\circ\varphi$. 
        It follows from $G$ being countable that the union of the $E_s$ is again a $\mu$-null set. 
        Hence, we can take $M$ as the $G$-invariant, $\mu$-null subset of $X$ that contains $M_0$ and the union of the $E_s$. 
        We take $N$ as the $H$-invariant, $\nu$-null subset of $Y$ such that $\theta_0$ restricts to a bijection $\theta\colon X\setminus M\to Y\setminus N$ where, for all $s\in G$, $\theta\circ\sigma_s = \tau_{\psi(s)}\circ \theta$ holds pointwise. 
        Finally, $\theta$ preserves measure, since for all measurable $E\subseteq X\setminus M$ we have that $\nu(\theta(E)) = I_\nu (\mathbbm{1}_{\theta(E)}) = I_\nu(\varphi(\mathbbm{1}_E)) = I_\mu(\mathbbm{1}_E) = \mu(E)$.
    \end{proof}

    Given a probability-measure-preserving action $\actionof{G}{(X,\mu)}$ on a standard probability space, there is a canonical way to extend $\mu$ to a trace in $\lone{G}{(X,\mu)}$. 
    \begin{definition}\label{canonicalTrace}
        Let $\actionof{G}{(X,\mu)}$ be a probability-measure-preserving action on a standard probability space $(X,\mu)$. 
        The \textit{canonical trace} $T_\mu$ on $\lone{G}{(X,\mu)}$ extending $\mu$ is the linear functional $T_\mu\colon\lone{G}{(X,\mu)}\to\C$ given by 
        $$T_\mu(f\delta_s) = 
        \begin{cases}
            \int_X fd\mu,  &s=e \\
            0, &s\neq e
        \end{cases}$$
for all $f\in C(X)$ and all $s\in G$.
    \end{definition}
    
    \begin{theorem}\label{thmIsomMeasure}
        Let $\actionof{G}{(X,\mu)}$ and $\actionof{H}{(Y,\nu)}$ be probability-measure-preserving actions of discrete, countable groups $G$ and $H$ on standard probability spaces $(X,\mu)$ and $(Y,\nu)$.  
        Then $\actionof{G}{(X,\mu)}$ and $\actionof{H}{(Y,\nu)}$ are measure-isomorphic if and only if \mbox{$\lone{G}{(X,\mu)}$} and $\lone{H}{(Y,\nu)}$ are isometrically isomorphic via an isometric isomorphism intertwining $T_\mu$ and $T_\nu$. 
    \end{theorem}
    \begin{proof}
        Assume that $\actionof{G}{(X,\mu)}$ and $\actionof{H}{(Y,\nu)}$ are measure-isomorphic. Using \autoref{measureIsomOnAlg} find a group isomorphism $\psi\colon G\to H$ and an isomorphism $\varphi\colon\Linfty\to L^\infty(Y,\nu)$ such that $I_\nu\circ\varphi = I_\mu$ and $\varphi\circ\alpha_s = \beta_{\psi(s)}\circ\varphi$ for all $s\in G$. 
        Define $\Phi\colon\lone{G}{(X,\mu)}\to\lone{H}{(Y,\nu)}$ by 
        $$\Phi(f\delta_s) = \varphi(f) \delta_{\psi(s)},$$
        for all $f\in\Linfty$ and $s\in G$. 
        It is routine to check that $\Phi$ is an isometric isomorphism. As $I_\nu\circ\varphi = I_\mu$, it is clear that $T_\nu\circ\Phi = T_\mu$.

        For the converse, let $G\curvearrowright^\alpha\Linfty$ and $H\curvearrowright^\beta L^\infty(Y,\nu)$ denote the induced actions. 
        If $\Phi\colon\lone{G}{(X,\mu)}\to\lone{H}{(Y,\nu)}$ is an isometric isomorphism, it follows from \autoref{lemaMeasureIsom} that $\varphi = \Phi|_{\Linfty}$ is an isomorphism from $\Linfty$ to $L^\infty(Y,\nu)$ and that there exists a group isomorphism $\psi\colon G\to H$ such that $\varphi\circ\alpha_s = \beta_{\psi(s)} \circ\varphi$ for all $s\in G$.
        Since $T_\nu\circ\Phi = T_\mu$ we have that $\nu\circ\varphi = \mu$.
        Thus the result follows from \autoref{measureIsomOnAlg}.
    \end{proof}

\end{document}